\documentclass[a4paper,12pt]{amsart}
\usepackage{amsthm}
\usepackage{amsmath}
\usepackage{amssymb}
\usepackage{mathtools}
\usepackage{mathpazo}
\usepackage{color}

\def\ra{{\rightarrow}}
\def\loc{{\rm loc}}

\DeclareMathOperator{\Ker}{Ker}

\theoremstyle{plain}
\newtheorem{thm}{Theorem}[section]

\newtheorem{lemma}[thm]{Lemma}

\newtheorem{ass}[thm]{Assertion}

\begin{document}

\title{There exist no contact Anosov diffeomorphisms}

%
\author[M Asaoka]{Masayuki Asaoka}
\address{Faculty of Science and Engineering, Doshisha University,
 1-3 Tatara Miyakodani, Kyotanabe 610-0394, Japan.}
\email{masaoka@mail.doshisha.ac.jp}
%

%
\author[Y Mitsumatsu]{Yoshihiko Mitsumatsu}
\address{Department of Mathematics, Chuo University, 
1-13-27 Kasuga, Bunkyo-ku, 
Tokyo 112-8551, Japan.}
\email{yoshi@math.chuo-u.ac.jp}
%

\thanks{
The first author is 
supported in part by JSPS KAKENHI Grants 22K03302. 
The second author is 
supported in part by JSPS KAKENHI Grants 
21K18579, 21H00985, 23K20798, and 
by Chuo University Grant for Special Research.}



\begin{abstract}
For any Anosov diffeomorphims 
on a closed odd dimensional manifold, 
there exists no invariant 
contact structure. 
\end{abstract}

\maketitle 

\section{Result}
In this short note we prove the following.  
\begin{thm}
\label{thm:main}
 Let $M$ be an odd dimensional closed manifold
 and $f$ a $C^2$ Anosov diffeomorphism on $M$.
Then, $f$ preserves no contact structures on $M$.
\end{thm}

There are many Anosov flows,  instead of considering Anosov diffeomorphisms, 
which preserve contact structures, where the strong stable and unstable leaves are 
Legendrian submanifolds. In dimension 3, 
now we know there are a lot of Anosov  Reeb flows even on hyperbolic 3-manifolds \cite{FH}
and in the 
higher dimensional case, the geodesic flows of negatively curved 
manifolds yield Reeb flows 
which are Anosov.   
As we will see, because of the dimension of the stable or unstable foliations, 
there exist no contact Anosov diffeomorphisms. 

The question asking if there exists such a diffeomorphism 
is also motivated by the following two problems on contact diffeomorphims. 
The first one is the problem of existence of hypersurfaces 
in contact manifolds of dimension 5 or higher 
which are
$C^r$-approximated by convex hypersurfaces. 
For the 3-dimensional case, it is always affirmative 
in any regularity up to $r=\infty$ \cite{G}, 
while in dimension 5 and higher, 
by Honda-Huang \cite{HH} and Eliashberg-Pancholi \cite{EP}, 
any closed hypersurface in 
a contact manifold can be $C^0$-approximated 
by Weinstein convex hypersurfaces.  
If there existed an Anosov contact diffeomorphism $\varphi$, 
its mapping torus $M_\varphi$ 
in the contact manifold 
$(-\varepsilon, \varepsilon)\times M_\varphi$ 
would not be 
smoothly approximated by convex surfaces. 
Chaidez \cite{C} constructed such non-$C^2$-approximable hypersurfaces 
deeply using hyperbolic dynamics. 
 
Another problem is asking 
whether 
a given diffeomorphism 
of an odd dimensional closed manifold 
is approximated by contact diffeomorphisms of some contact structures. 
The result of the present paper implies any Anosov diffeomorphisms 
is not $C^1$-approximable in this sense, 
because the set of Anosov diffeomorphisms 
on a closed manifold is an open set in 
the $C^1$-topology 
in the group of diffeomorphims. 

\section{Proof}
Throughout the paper, let $f$ be a $C^{2}$-Anosov diffeomorphism 
on a closed manifold $M$ of dimension $2m+1$,   
 and $TM=E^s \oplus E^u$ be the Anosov splitting of $f$.
Take $0<\lambda<1$ and a $C^0$ Riemannian metric on $M$ such that
\begin{itemize}
 \item $E^s$ and $E^u$ are orthogonal with respect to the metric,
 and
 \item the norm $\|\cdot\|$ induced by the metric satisfies
 that $\|Df|_{E^s}\|\leq \lambda$ and $\|Df^{-1}|_{E^u}\| \leq \lambda$.
\end{itemize}
Put $m^s=\dim E^s$ and  $m^u=\dim E^u$.

\begin{lemma}
\label{lemma:InvariantDistribution}
Let $\zeta\subset TM$  be a continuous distribution of rank $\ell$ on $M$ 
which is invariant under the action of $f$. 
If at a point 
we have 
$p\in M$ $E^{s}(p)\subset \zeta(p)$ 
$[{\rm resp.} \,\, E^{s}(p)\supset \zeta(p) ]$,  
then, 
at any point $x$ in the stable manifold $W^{s}(p)$ of $p$ 
{\rm (}{\em i.e.}, 
the leaf of the stable foliation through $p${\rm )} 
we also have 
$E^{s}(x)\subset \zeta(x)$ 
$[{\rm resp.} \,\, E^{s}(x)\supset \zeta(x) ]$.  
The similar statement for $\zeta$ and $E^{u}$ is also true. 
\end{lemma}
\begin{proof}
We will prove in the case $m^{s}\leq\ell$ because for the other case  
a similar but simpler 
argument will suffice. 
Put
\begin{align*}
 Z=\{x \in M \mid E^s(x) \subset \zeta(x)\}.  
\end{align*}
$Z$ is a closed invariant set.  
For $x \in M$ and $\theta \geq 0$, put
\begin{align*}
C^s(x,\theta)
 & =\{v+w \mid v \in E^s(x), w \in E^u(x), \|w\| \leq \theta\|v\|\},\\
C^u(x,\theta)
 & =\{v+w \mid v \in E^s(x), w \in E^u(x), \|v\| \leq \theta\|w\|\}.
\end{align*}
Remark that $Df(C^u(x,\delta)) \subset C^u(f(x), \lambda^2\delta)$.
For a point $x\in M$,   
$E^s(x) \subset \zeta(x)$  is equivalent to  
for any $v \in E^{s}(x)$ and $\theta>0$
there exists some $v+w \in C^s(x,\theta)\cap \zeta(x)$ 
with $w\in E^u(x)$. 
Since $\zeta$ and $E^s$ are continuous subbundles of $TM$ 
and $Z$ is a compact invariant set, 
for any $\theta>0$ there exists $\delta>0$ such that 
any point $q$ of the $\delta$-neighborhood $W_\delta$ of $Z$ 
and for any $v \in E^{s}(q)$ 
there exists some $v+w \in C^s(q,\theta)\cap \zeta(q)$. 

Now fix any point $x\in W^{s}(p)$ and $v\in E^{s}(x)$. 
As $d(f^{k}(x), f^{k}(p))\ra 0$ when $k\ra \infty$ and $f^{k}(p)\in Z$, 
for $\theta=1$, 
there exists $K$ such that $f^{k}(x)\in W_\delta$ for any $k\geq K$ 
and therefore  we can find 
$Df^{k}(v)+w \in C^s(f^{k}(x),1)\cap \zeta(f^{k}(x))$. 
This implies $v+ Df^{-k}(w) \in C^s(x, \lambda^{-2k})\cap \zeta(x)$. 
As we can take $k$ arbitrarily large, this completes the proof.  
\end{proof}

Remark that in the above argument if $p\in M$ is 
a fixed or periodic point of $f$, the proof becomes 
easier. 
It is known that any codimension one Anosov diffeomorphism  
is transitive and thus the periodic points are dense \cite{N}.   
In that case, the whole proof can be simplified.  
In general, the transitivity is an open problem. 
\bigskip

We start to prove Theorem \ref{thm:main} by contradiction. 
Suppose that the Anosov diffeomorphism $f$ preserves 
 a contact structure $\xi$.  
Including the case where $\xi$ is not coorientable, 
we assume that a global pair of smooth 1-forms 
$\{\alpha, -\alpha\}$ defines $\xi$.  
 Put
\begin{align*}
S=\{p \in M \mid E^{s\,}(p) \subset \xi(p)\}, 
\\
U =\{p \in M \mid E^u(p) \subset \xi(p)\}. 
\end{align*}

\begin{ass}
\label{prop:volume}$S$ and $U$ are not empty. 
\end{ass}
\begin{proof}
As the proof is the same for $S$ and $U$ if we replace $f$ with $f^{-1}$, 
we show $S$ is not empty. By contradiction, assume $S$ is empty, 
namely, $E^s(p) \not\subset \xi(p)$ for any $p \in M$.
Let $\eta(p)$ be the orthogonal complement of $\xi(p) \cap E^s(p)$
 in $E^s(p)$ for any $p \in M$.
Then $\eta$ is a continuous one-dimensional subbundle of $E^s$
 such that $E^s(p)=\eta(p) \oplus (\xi \cap E^s)(p)$
 and $T_p M=\eta(p) \oplus \xi(p)$.
For $n \geq 1$ and $p \in M$,
 there exists
 a linear isomorphism 
 $a_n: \eta(p) \ra \eta(f^n (p))$
 such that the $Df^n|_{E^s(p)}:E^s(p) 
 \ra E^s(f^n (p))$ 
 has the form
\begin{equation*}
\begin{bmatrix} a_n(p) & 0 \\ * & * \end{bmatrix},
\end{equation*}
 with respect to the decomposition $\eta \oplus (\xi \cap E^s)$.
Since $\|Df^n|_{E^s}\| \leq \lambda^n$,
 we have $\|a_n(p)\|\leq \lambda^n$ for any $p \in M$.
Since $\eta$ is transverse to $\xi=\Ker \alpha$ and $M$ is compact,
 there exists $C>1$ such that
 $C^{-1}\|v\| \leq |\alpha(v)| \leq C\|v\|$ for any $v \in \eta$.
Then, for $p \in M$ and $v \in \eta(p)$,
\begin{align*}
 |(f^n)^*\alpha(v)|= |\alpha(Df^n v)|=|\alpha(a_n(p) v)|
 \leq C \|a_n(p)v\| \leq C\lambda^n\|v\|
 \leq C^2\lambda^n |\alpha(v)|.
\end{align*}
In particular, there exists $N \geq 1$ such that
\begin{equation*}
 |(f^N)^*\alpha(v)| \leq \frac{1}{2}|\alpha(v)|  
\end{equation*}
 for any $v \in \eta$.

Since $\xi=\Ker \alpha$ is $Df$-invariant,
there exists a positive smooth function $\lambda_N$ on $M$
such that $\pm (f^N)^*\alpha =\pm\lambda_N \cdot \alpha$.
Then, the above inequality implies that $0 < \lambda_N \leq 1/2$.
Since
\begin{align*}
 (f^N)^*d\alpha
 & = d((f^N)^*\alpha)=\pm d(\lambda_N \cdot \alpha)
 =\pm (d\lambda_N \cdot \alpha +\lambda_N \cdot d\alpha),
\end{align*}
we have
\begin{equation*}
 (f^N)^*(\alpha \wedge (d\alpha)^m)
=\pm \lambda_N^{m+1}\alpha \wedge (d\alpha)^m.
\end{equation*}
By compactness of $M$, the total volume 
 of the 
 measure 
 $|\alpha \wedge (d\alpha)^m|$ is finite and
\begin{equation*}
\int_M |\alpha \wedge (d\alpha)^m|
 =\int_M |(f^N)^*(\alpha \wedge (d\alpha)^m)|
 =\int_M |\lambda_N^{m+1}\alpha \wedge (d\alpha)^m|.
\end{equation*}
This contradicts that $|\lambda_N|\leq 1/2$.

\end{proof}
 
Let us complete the proof of Theorem \ref{thm:main}.  
By replacing $f$ with its inverse if necessary,
we may assume that $m^s\geq m+1$.

We apply Lemma \ref{lemma:InvariantDistribution} 
to the contact structure $\xi$ 
and therefore $S$ plays the 
role of $Z$ in the Lemma.   
From the above Assertion, we can take a point $p \in S$.

By the stable manifold theorem, (cf. e.g., Theorem 6.2 in \cite{S})
the local stable manifold $W^s_{\loc}(p)$
is an embedded $C^r$-disk of dimension $\dim E^s$.
For $v,w \in E^s(p)$, there exist $C^r$-vector fields $X,Y$ on $M$
 such that $X(p)=v$, $Y(p)=w$, and
 $X(q),Y(q) \in T_q W^s_{\loc}(p)=E^s(q)$
 for any $q \in W^s_{\loc}(p)$.
Then, we have $[X,Y](p) \in T_p W^s_{\loc}(p)=E^s(p)$.
Since $E^s(p) \subset \xi(p)=\Ker \alpha_p$, we obtain
\begin{align*}
 d\alpha_p(v,w)=X(\alpha(Y))(p)-Y(\alpha(X))(p)-\alpha([X,Y])(p)=0.
\end{align*}
In particular, the two-form $d\alpha$ vanishes on $E^s(p)$.
It contradicts that $d\alpha$ is non-degenerate on $\xi(p)$
 since $\dim \xi=2m$ and $\dim E^s \geq m+1$.  
 This completes the proof of Theorem \ref{thm:main}. 
 \hfill $\square$


\begin{thebibliography}{99}

\bibitem[C]{C}
{Chaidez, Julian}; 
{\it 
Robustly non-convex hypersurfaces in contact manifolds}, 
preprint, arXive:2406.05979v2, 2024.



\bibitem[EP]{EP}
{Eliashberg, Yacov \&   Pancholi, Dishant}; 
{\it 
Honda-Huang's work on contact convexity revisited
}, 
preprint, 
arXiv:2207.07185, 2022.



\bibitem[FH]{FH}
{Foulon, Patrick \& Hasselblatt, Boris}; 
{\it 
Contact Anosov flows on hyperbolic 3-manifolds
}, 
Geom. Topol. {\bf 17} no. 2,  (2013),1225--1252.

\bibitem[G]{G}
{Giroux, Emmanuel}; 
{\it 
Convexity in contact topology
}, 
Comment. Math. Helv, {\bf 66} no.4, (1991) 637--677.

\bibitem[HH]{HH}
{Honda, Ko \& Huang, Yang}; 
{\it Convex hypersurface theory in contact topology},  
preprint, arXiv:1907.06025, 2019.


\bibitem[N]{N}
{Newhouse, Sheldon E. }; 
{\it 
On codimension one Anosov diffeomorphisms
}, 
Amer. J. Math., {\bf 92}, (1970) 761--770.




\bibitem[S]{S}{Shub, Michael}; 
{\it Global stability of dynamical systems},  
Springer-Verlag, New York, 1987.  



\end{thebibliography}
\end{document}